\documentclass{amsart}

\usepackage[shortalphabetic]{amsrefs}
\usepackage{amsmath}
\usepackage{amsthm}
\usepackage{amssymb}
\usepackage{graphicx}
\usepackage{epstopdf}
\usepackage{comment}
\usepackage{booktabs}
\usepackage{subcaption}
\usepackage{url}

\newtheorem{thm}{Theorem}

\newtheorem{definition}[equation]{Defintion}

\theoremstyle{definition}

\DeclareMathOperator{\End}{End}

\DeclareMathOperator{\Adj}{Adj}
\DeclareMathOperator{\Der}{Der}
\DeclareMathOperator{\Cent}{Cent}

\DefineSimpleKey{bib}{myurl}

\newcommand\myurl[1]{\url{#1}}

\BibSpec{webpage}{%
  +{}{\PrintAuthors} {author}
  +{,}{ \textit} {title}
  +{}{ \parenthesize} {date}
  +{,}{ \myurl} {myurl}
  +{,}{ } {note}
  +{.}{ } {transition}
}

\author{Joshua Maglione}\address{
	Department of Mathematics\\
	Colorado State University\\
	Fort Collins, CO 80523\\
	USA
}
\email{maglione@math.colostate.edu}
\title{Automorphisms of small prime power groups}
\date{\today}

\begin{document}

\begin{abstract}
If $f(p,n)$ is the number of pairwise nonisomorphic groups of order $p^n$, and $g(p,n)$ is the number of groups of order $p^n$ whose automorphism group is a $p$-group, then, for $n\leq 7$, we prove that the ratio $g(p,n)/f(p,n)$ is bounded away from 1 as the prime $p$ grows to infinity.
In addition, we provide some data on the number of groups whose automorphism group is a group of prime power order, for primes no larger than 11. 
\end{abstract}

\maketitle

\section{Introduction}

Within the last decade, we have gained incredible knowledge about the automorphism groups of $p$-groups. 
Nevertheless, we are still unable to answer basic questions like the following found in \cite{M:Survey}. 
If $f(p,n)$ is the number of pairwise nonisomorphic groups of order $p^n$ and $g(p,n)$ the number of groups of order $p^n$ whose automorphism group is a $p$-group, then does
\begin{equation}\label{eqn:question} 
\lim_{n\rightarrow \infty} \dfrac{g(p,n)}{f(p,n)}=1?
\end{equation}

A recent result of Martin and Helleloid-Martin is summarized as follows. The automorphism group of almost all $p$-groups is a $p$-group \cite{HM:Aut}*{Theorem 1}. 
However, Helleloid-Martin's perspective differs from the perspective of Mann's question in (\ref{eqn:question}).
Instead, they consider three parameters: the prime, the number of generators, and the $p$-class.
They prove three limits by fixing two parameters and allowing one to vary, and they provide a list of boundary conditions for which their theorem need not hold.
Here, we verify, in some sense, the necessity of some of those boundary conditions with the following theorem.

\begin{thm}\label{thm:main}
Let $f(p,n)$ be the number of pairwise nonisomorphic groups of order $p^n$ and $g(p,n)$ the number of groups of order $p^n$ whose autmorphism group is a $p$-group. 
If $n\leq7$, then
\[ \lim_{p\rightarrow\infty} \dfrac{g(p,n)}{f(p,n)} \leq \frac{2}{3}.\]
\end{thm}

A sharper upper bound likely exists --- certainly, for individual $n$ values. 
The interest is not in the limit itself, but in the fact that this limit is bounded away from 1. 
Thus, we do not get a statement like that of Helleloid-Martin for small $p$-groups. 
In fact, it may be that the majority of these automorphism groups are not $p$-groups.

In addition to Theorem \ref{thm:main}, we provide data on the number of automorphism groups of various small $p$-groups for $p\leq 11$, which expands on the known values, $p\leq 5$ and $n\leq 7$, as seen in \cite{HM:Aut}*{p. 295}.
In order to make the computation quicker, we use techniques described in \citelist{\cite{W:Char}\cite{M:classical}} to find more characteristic subgroups outside of the standard verbal and marginal ones. 
For the majority of these groups, we find many more characteristic subgroups. 
In order to carry out these computations for groups of order $11^7$, we use ideas in \cite{BMW:genus2}. 
Indeed, the groups of genus 1 and 2 are still quite an obstacle; some of whose automorphism group requires several hours to compute in {\sc Magma} V2.21-1 on an Intel Xeon W3565 at 3.20 GHz.
These tables require about 8 months of computation even with the state of the art algorithms. 

\section{Preliminaries}\label{sec:filters}

\subsection{Notation} We denote the set of nonnegative integers by $\mathbb{N}$, and the set of all subsets of a set $G$ by $2^G$. For groups and rings, we follow standard notation found in \cite{G:GT}. For $g,h\in G$, we set 
\[ [g,h]=g^{-1}g^{h}=g^{-1}h^{-1}gh;\] 
for $X,Y\subseteq G$, we set 
\[ [X,Y]=\langle [x,y] : x\in X, y\in Y\rangle.\]
We let $\mathbb{Z}_p$ denote the group $\mathbb{Z}/p\mathbb{Z}$.

For a $p$-group $G$, the lower central series of $G$ is defined recursively with $\gamma_1(G)=G$ and $\gamma_{i+1}(G)=[\gamma_i(G),G]$. 
Similarly, the exponent $p$-central series of $G$ is defined by $\eta_1(G)=G$, and $\eta_{i+1}(G) = [\eta_i(G),G]\eta_i(G)^p$. 
If $\gamma_c(G)\ne 1$ and $\gamma_{c+1}(G)=1$, then $G$ has (nilpotency) \emph{class $c$}. 
Similarly, if $\eta_c(G)\ne 1$ and $\eta_{c+1}(G)=1$, then $G$ has \emph{$p$-class $c$}.
Suppose $P$ and $G$ are $p$-groups with $p$-class $c$ and $c+1$ respectively. 
We say $G$ is an \emph{immediate descendant} of $P$ if $G/\eta_{c+1}(G) \cong P$. 
If $G$ is $p$-class 2, of order $p^n$, and $\log_p [G:\Phi(G)]=d$, then, for our purposes, we say the \emph{genus} of $G$ is $n-d$. 
A more general definition of genus in the context of groups is described in \cite{BMW:genus2}.

\subsection{Filters}

We use filters to gain exponential improvements in computing the automorphism groups of some small $p$-groups.
In \cite{W:Char}, J.\! B.\! Wilson introduced filters as a generalization of an $N$-series, defined by Lazard in \cite{L:Nseries}. 
The appeal of filters is to have access to an associated graded Lie ring, and hence, they are easy to compute with on account of capuring linear structure.

\begin{definition}
Let $\langle M,0,+,\preceq\rangle$ be a commutative monoid with pre-order, and $G$ a group. A \emph{filter} is a function $\phi$ from $M$ into the subgroups of $G$ satisfying the following conditions for all $s,t\in M$:
\begin{enumerate}
\item $[\phi_s,\phi_t] \leq \phi_{s+t}$, and
\item $s\preceq t$ implies $\phi_s\geq \phi_t$.
\end{enumerate}
\end{definition}

Every filter $\phi:M \rightarrow 2^G$ induces a \emph{boundary filter} $\partial\phi: M \rightarrow 2^G$ where $\partial\phi_s = \langle \phi_{s+t} | t\in M-\{0\}\rangle$. 
With this, we can define a graded Lie ring for $G$.
Define $L(\phi) = \bigoplus_{s\in M} \phi_s / \partial\phi_s$, where $L_0=0$, with multiplication given by
\[ [ \partial\phi_s x, \partial\phi_t y ] = \partial\phi_{s+t} [x,y]. \] 

\begin{thm}[\cite{W:Char}*{Theorem 3.1}]
$L(\phi)$ is an $M$-graded Lie ring.
\end{thm}

As the filter refines so does its associated Lie ring. 
Our aim is to construct long characteristic series with an associated graded Lie algebra. We assume our monoids are $\mathbb{N}^d$ for some $d$. 
Provided we can find a new characteristic subgroup to add to our filter, we can use this new subgroup to generate more characteristic subgroups throughout the filter. 
Indeed, suppose $\pi : X\rightarrow 2^G$ satisfies 
\begin{enumerate}
\item $0\in X\subseteq \mathbb{N}^d$ and $\langle X \rangle = \mathbb{N}^d$;
\item if $s\in X$ and $t\in \mathbb{N}^d$ with $t\preceq s$, then $t\in X$;
\item for all $s\in X$, $\pi_s\trianglelefteq G$;
\item if $s,t\in X$ with $s\preceq t$, then $\pi_s\geq \pi_t$.
\end{enumerate}
For $s\in \langle X\rangle$, a \emph{partition} of $s$ with respect to $X$ is a sequence $(P_1,...,P_k)$ where each $P_i\in X$ and $s = \sum_{i=1}^kP_i$. 
Let $\mathcal{P}_X(s)$ denote the set of partitions of $s\in \langle X\rangle$ with respect to $X$, and if $P=(P_1,...,P_k)\in\mathcal{P}_X(s)$, then set
\[ [\pi_P ] = [ \pi_{P_1}, ... ,\pi_{P_k} ].\] 
We generate a filter $\overline{\pi} : \langle X\rangle \rightarrow 2^G$ defined as follows
\[ \overline{\pi}_s = \prod_{P\in\mathcal{P}_X(s)} [\pi_P] ,\]
cf. \cite{W:Char}*{Theorem 3.3}.

We are back to the problem of finding characteristic subgroups in a $p$-group. 
Suppose we start with the filter $\eta:\mathbb{N}\rightarrow 2^G$ given by the exponent $p$-central series of $G$. 
Then $L(\eta)$ has an associated $\mathbb{N}$-graded Lie algebra, which yeilds $\mathbb{Z}_p$-bilinear maps from the graded product (eg.\!\! $[,] : L_s\times L_t \rightarrowtail L_{s+t}$). 
We turn to some associated algebras for these bilinear maps. 
Suppose $\circ : U \times V\rightarrowtail W$ is a biadditive map of abelian groups; define the adjoint, centroid, and derivation rings as
\begin{align*}
\Adj( \circ ) &= \{ (f,g) \in \End(U)\times \End(V)^{\text{op}} : \forall u\in U, \forall v\in V, uf \circ v = u \circ gv \},\\
\Cent( \circ ) &= \{ (f,g,h) \in \End(U)\times \End(V)\times \End(W) : \forall u\in U, \forall v\in V, \forall w\in W, \\
&\qquad uf \circ v = u \circ vg = (u\circ v)h \}, \text{ and }\\
\Der( \circ ) &= \{ (f,g,h) \in \mathfrak{gl}(U)\times \mathfrak{gl}(V)\times \mathfrak{gl}(W) : \forall u\in U, \forall v\in V, \forall w\in W, \\
&\qquad uf \circ v + u \circ vg = (u\circ v)h \}.
\end{align*}

It is in these rings we find more characteristic structure in $G$ \cite{W:Char}*{Section 4}. Indeed, the Jacobson radical acts on the homogeneous components and yields characteristic subgroups (for $\Der(\circ)$, this is done in the associative enveloping algebra). 

\section{Data}

We provide data on the number of $p$-groups whose automorphism group is also a $p$-group in Tables \ref{fig:aut_p-grp} and \ref{fig:aut_p-grp2}. 
We examine the groups of order 512, and for each group, we record the the change in length of the filter and the change in order of the largest factor. 
The longer the filter, or the smaller the largest factor, the faster we can compute the automorphism groups. 
Our data is found in Figures \ref{fig:Lengths} and \ref{fig:MaxSections}. 
Some data on the kinds of refinements which seem most successful can be found in \cite{W:filters}.
All of these computations were done in {\sc Magma} \cite{Magma}.

\begin{table}[h]
\begin{tabular}{p{2cm}p{2.5cm}p{2.5cm}p{3cm}}
  \toprule
  Order  & $p^5$    & $p^6$ & $p^7$ \\
  \midrule
  $p=2$  & 36 of 51 & 211 of 267  & 2,067 of 2,328 \\
  $p=3$  & 0 of 67  & 30 of 504   & 2,119 of 9,310 \\
  $p=5$  & 1 of 77  & 65 of 674   & 11,895 of 34,297 \\
  $p=7$  & 0 of 83  & 91 of 860   & 42,208 of 113,147 \\
  $p=11$ & 1 of 87  & 189 of 1,192 & 286,385 of 750,735 \\
  \bottomrule
\end{tabular}
\caption{The number of $p$-groups whose automorphism group is a $p$-group.}
\label{fig:aut_p-grp}
\end{table}

\begin{table}[h]
\begin{tabular}{p{2cm}p{4cm}p{4cm}}
\toprule
Order  & $p^8$    & $p^9$ \\ 
\midrule
$p=2$  & 54,463 of 56,092 & 10,477,331 of 10,494,213 \\
$p=3$  & 1,002,258 of 1,396,077 & \\
\bottomrule
\end{tabular}
\caption{The number of $2$-groups whose automorphism group is a $2$-group.}
\label{fig:aut_p-grp2}
\end{table}

The timing of the algorithm for computing the automorphism group of a $p$-group is heavily dependent on the the orders of the factors of the characteristic series it works with \citelist{\cite{O:p-grp-alg}\cite{ELGO:Auts}}. 
In fact, the algorithm is exponential in the order of the largest factor. 
Therefore, Figure \ref{fig:Lengths} is an incomplete picture of how the rings in Section \ref{sec:filters} aid in the computation.
On the other hand, we see that, in Figure \ref{fig:MaxSections}, about $80\%$ of groups of order 512 have an exponential speed-up. 
Thus, without these improvements, the computation of the automorphism groups would be orders of magnitude harder.

\begin{figure}[h]
\begin{subfigure}[b]{.5\linewidth}
\centering
\includegraphics[width=.95\linewidth]{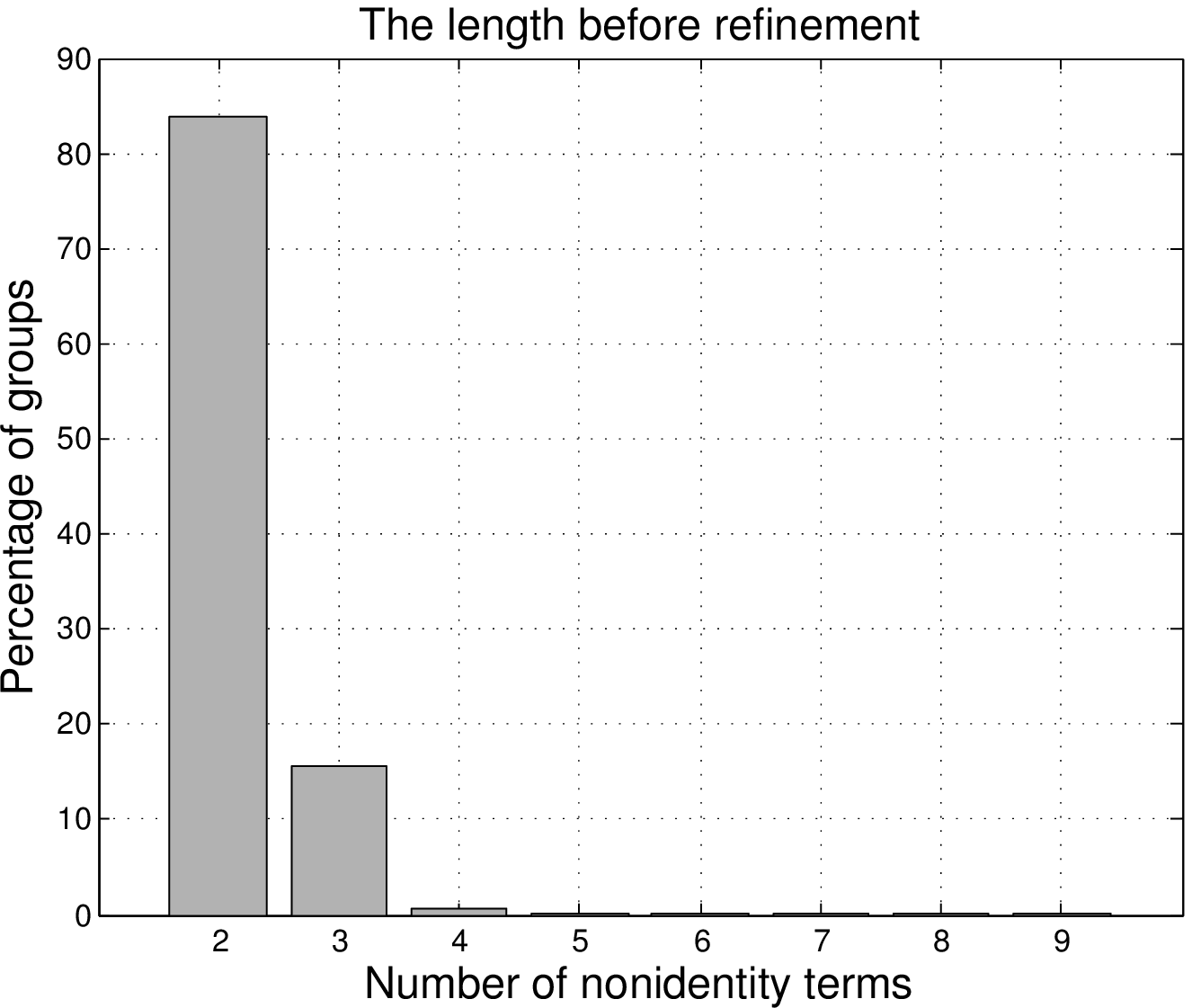}
\subcaption{Exponent $p$-central series.}
\end{subfigure}%
\begin{subfigure}[b]{.5\linewidth}
\centering
\includegraphics[width=.95\linewidth]{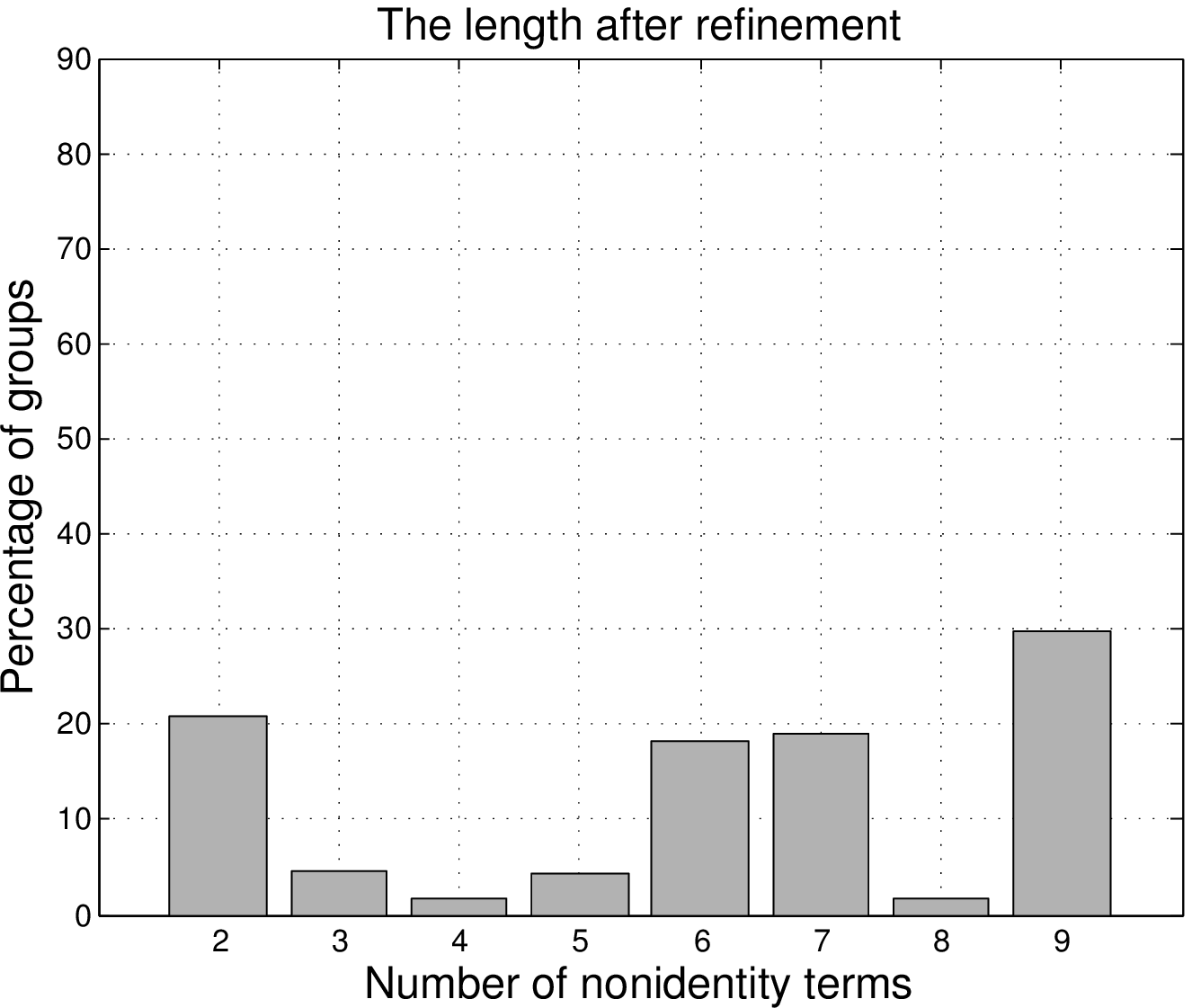}
\subcaption{Fully refined filter.}
\end{subfigure}
\caption{The bar graphs show how long a characteristic series gets by refining the $p$-central series using $\Adj(\circ)$, $\Cent(\circ)$, and $\Der(\circ)$ for groups of order 512. The maximum possible is 9.}\label{fig:Lengths}
\end{figure}

\begin{figure}[h]
\begin{subfigure}[b]{.5\linewidth}
\centering
\includegraphics[width=.95\linewidth]{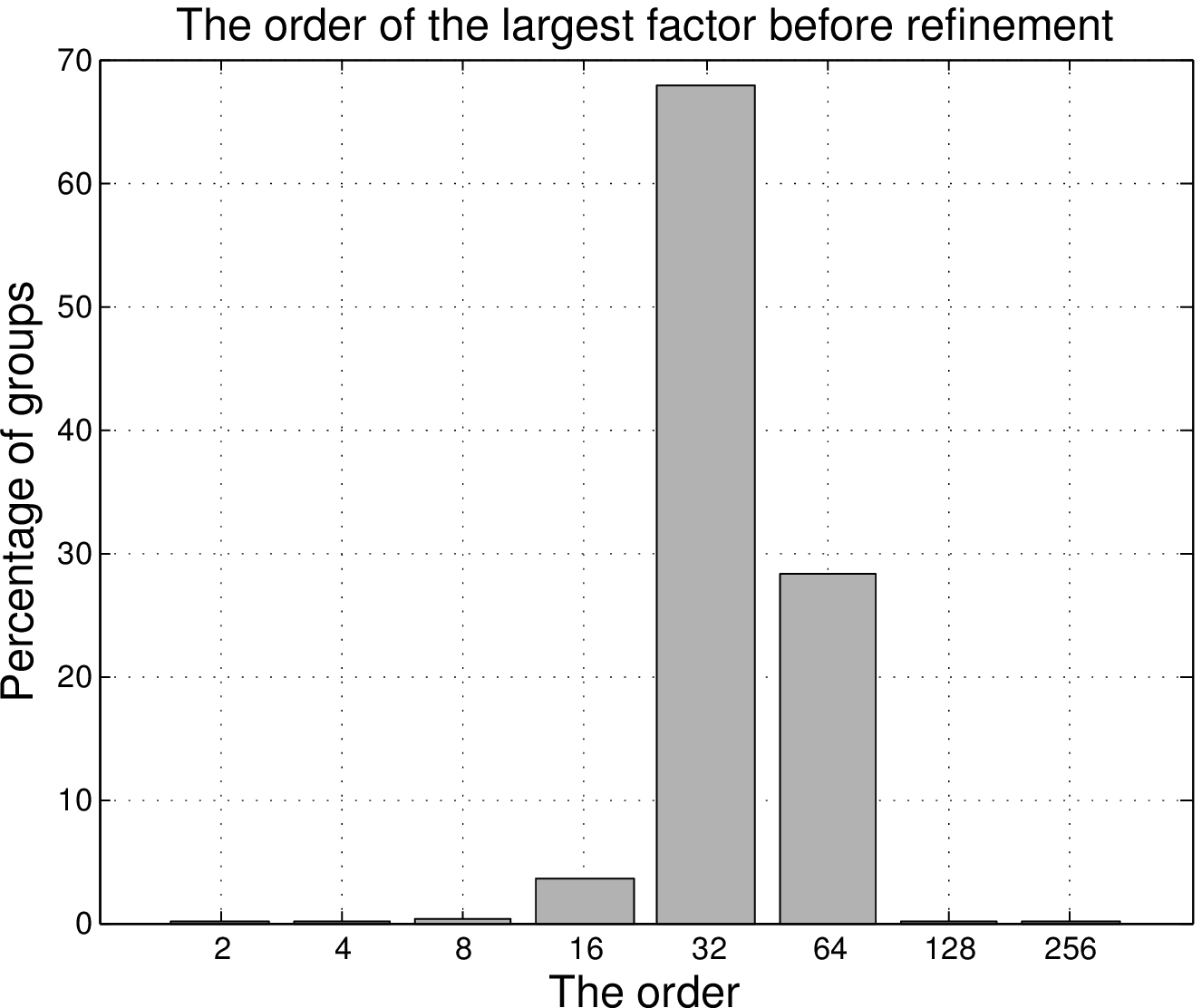}
\subcaption{Exponent $p$-central series.}
\end{subfigure}%
\begin{subfigure}[b]{.5\linewidth}
\centering
\includegraphics[width=.95\linewidth]{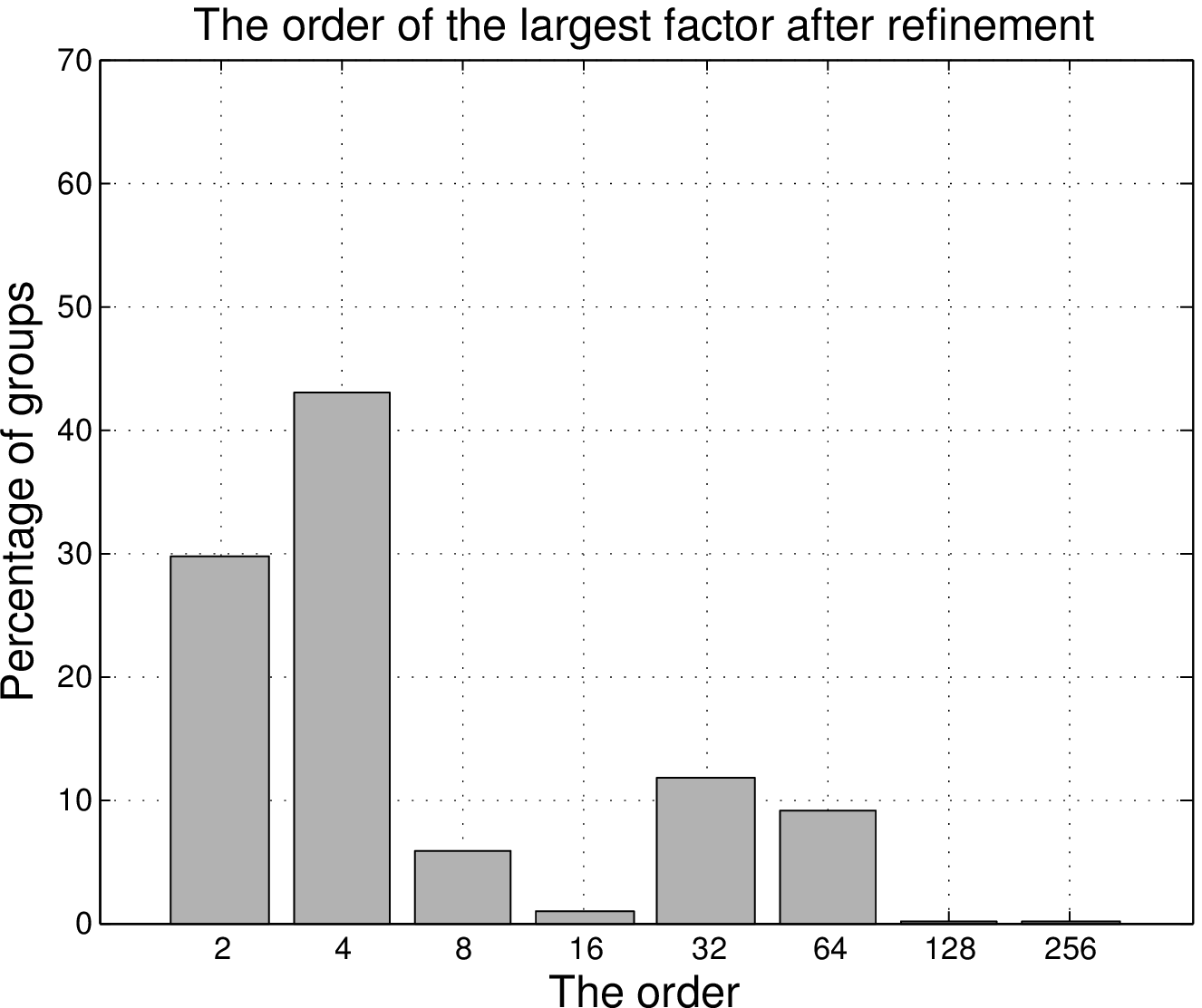}
\subcaption{Fully refined filter.}
\end{subfigure}
\caption{The bar graphs show the orders of the largest factors in each of the series for groups of order 512. The larger the sections, the harder it is to compute the automorphism group.}\label{fig:MaxSections}
\end{figure}

\section{Proof of Theorem \ref{thm:main}}

We prove Theorem \ref{thm:main} by considering the possible values for $n$, while assuming $p$ is an odd prime. 
The statement of the theorem is immediate for $n\leq 4$ because, in that case, there is a constant number of groups of order $p^n$. 
For each $n\leq 4$, at least one third of the groups of order $p^n$ are abelian. Thus, for each of those groups, there exists an automorphism which is an involution.

\begin{proof}[Proof of Theorem \ref{thm:main}]
The theorem follows for $n\leq 4$, so we first consider when $n=5$. 
Suppose $G$ has order $p^5$ and $G'=G^p=Z(G)\cong\mathbb{Z}_p^2$. 
There exists generators $x,y,z$ for $G$ such that $G'=\langle [y,x],[z,x] \rangle$ and $[z,y]=1$. 
In addition, we assume that $G^p=\langle y^p, z^p\rangle$. 
There are at least $p+O(1)$ such groups \cite{VL:p^5}*{p. 12}. 
If $x^p\ne 1$, then there exists $a,b\in\mathbb{Z}_p$ such that $\overline{x}=xy^az^b$ and $\overline{x}^p=1$ because $G^p=\langle y^p,z^p\rangle$. 
Thus, without loss of generality, we assume $x^p=1$. The map fixing $x$ and inverting $y$ and $z$ induces an automorphism of $G$. 
Since there are $2p+O(1)$ groups of order $p^5$, the theorem follows for $n=5$.

Now suppose $n=6$, $p\geq 5$, and $G$ is a group of order $p^6$, which is an immediate descendant from 
\[ P = \langle x,y,z \mid [z,x] = [z,y] = 1, \text{class 2}, \text{exponent } p\rangle.\]
Nearly all groups of order $p^6$ are immediate descendents of $P$, but about one third of them satisfy these equations
\begin{equation}\label{set1}
\begin{aligned}
  \left[z,x\right] &= [y,x,x], \\
  [z,y] &= [y,x,z] = z^p = 1.
\end{aligned}
\end{equation}

The map which inverts $x$ and $y$ and fixes $z$ induces an automorphism of $G$. 
There are $p^2 + O(p)$ groups of order $p^6$ that satisfy (\ref{set1}) \cite{VL:p^6}*{p. 37}, and there are $3p^2+O(p)$ groups of order $p^6$ \cite{NOVL:p^6}.

Let $G$ be a group of order $p^7$ for $p\geq 5$, and suppose that $G$ is an immediate descendant of 
\[ Q = \langle x, y, z \mid [z,y] = 1, \text{class 2}, \text{exponent } p\rangle.\]
An abundant source of these descendants satisfy these sets of equations, as seen in \cite{OVL:Notes},
\begin{align}
1=[z,y]&=[y,x,z],\nonumber\\
[y,x,x]&=[z,x,z],\label{set2}\\ 
[z,x,x]&=[y,x,y].\nonumber
\end{align}

The map which inverts $x$, $y$, and $z$ induces an automorphism of $G$.
There are $p^5 + O(p^4)$ groups that satisfy (\ref{set2}) \cite{OVL:Notes}*{pp. 56 -- 57}. 
Since there are $3p^5 + O(p^4)$ groups of order $p^7$ for $p\geq 5$, the theorem follows.
\end{proof}


Because Helleloid-Martin consider a different ratio from $g(p,n)/f(p,n)$, Theorem \ref{thm:main} does not fit nicely into their framework. 
However, their result is the closest comparison to Theorem \ref{thm:main}. 
In fact, when they fix the number of generators and $p$-class but let the prime vary, they do not consider the families of groups in the proof of Theorem \ref{thm:main}, and
in this light, it seems reasonable to guess that the theorem of Helleloid-Martin might not hold for smaller values than they reported. 
Therefore, while improvements for the choice of values for the parameters might be possible, we suggest that they are likely close to optimal.

\section*{Acknowledgements}
The author is thankful to James B. Wilson and Alexander Hulpke for helpful feedback and to Eamonn A. O'Brien for his initial question.

\end{document}